\documentclass[12pt,reqno]{amsart}
\usepackage{amssymb}
\usepackage{upgreek}
\usepackage{dsfont }
\usepackage{mathrsfs}
\usepackage{mathtools}
\usepackage[all]{xy}
\usepackage{color}
\usepackage{verbatim}
\usepackage{tensor}
\usepackage{setspace} 
\allowdisplaybreaks
\newtheorem{thm}{Theorem}[section]

\theoremstyle{definition}

\theoremstyle{remark}
\newtheorem{remark}[thm]{Remark}

\newcommand{\af}{\alpha}
\newcommand{\bt}{\beta}
\newcommand{\gm}{\gamma}
\newcommand{\dt}{\delta}
\newcommand{\ep}{\varepsilon}

\newcommand{\om}{\omega}

\newcommand{\Gm}{\Gamma}

\newcommand{\Dt}{\Delta}

\newcommand{\R}{{\mathbb{R}}}
\newcommand{\C}{{\mathbb{C}}}

\newcommand{\Hess}{\mathrm{Hess}}
\newcommand{\nab}{\nabla}
\newcommand{\pr}{\partial}

\begin{document}


\title[Matrix estimate for the Green function on K\"{a}hler manifolds]
{A Matrix Li-Yau-Hamilton estimate for the Green function on K\"ahler manifolds}

\author[Jiewon Park]{Jiewon Park}
\address{
Department of Mathematics\\
Yale University\\
12 Hillhouse Ave, New Haven, CT 06511\\
United States} \email{jiewon.park\-@\-yale.\-edu}

\begin{abstract}   
In this paper we prove a matrix Li-Yau-Hamilton inequality for the Green function on complete K\"ahler manifolds with nonnegative holomorphic bisectional curvature. This estimate can be seen as an elliptic analogue of the matrix estimate of Cao and Ni for the heat equation on K\"ahler manifolds, or the complex analogue of the estimate for Riemannian manifolds obtained previously by the author.
\end{abstract} 

\maketitle

\section{Introduction}

In \cite{LY}, Li and Yau proved a fundamental gradient estimate for positive solutions to the heat equation on complete Riemannian manifolds with Ricci curvature bounded below. Since the Harnack inequality can be derived by integrating the inequality along geodesics, this estimate itself is often referred to as gradient Harnack inequality. In \cite{H2}, Hamilton showed the following generalization of the gradient Harnack inequality. Inequalities of this type are often called matrix Li-Yau-Hamilton estimates.

\begin{thm}[\cite{H2}, Main Theorem] \label{hamilton_thm}
Let $(M,g)$ be a compact Riemannian manifold and $f: M \to \R$ a positive solution to the heat equation on $M$. Suppose that $M$ has nonnegative sectional curvature and parallel Ricci curvature. Then for any $t>0$ and any vector field $V_i$ on $M$,
\begin{equation} \label{hamilton_heat}
	f_{ij} + \frac{f}{2t} g_{ij} + D_i f \cdot V_j + D_jf \cdot V_i + f V_i V_j \geq 0.
\end{equation}

\end{thm}

Although the drawback is that the curvature assumption is more restrictive than just $\mathrm{Ric} \geq 0$, Hamilton's estimate is on the full Hessian of $f$. By taking the trace of (\ref{hamilton_heat}) one recovers the Li-Yau gradient estimate.

An elliptic analogue of the Li-Yau gradient estimate is the gradient estimate for the Green function, obtained by Colding in \cite{C} for nonparabolic manifolds with nonnegative Ricci curvature. We recall that a Riemannian manifold of dimension $m$ is said to be {\it nonparabolic} if it possesses a positive Green function, that is, a smooth function $G: M \backslash D \to (0, \infty)$ so that $$\Delta G(x,\cdot) = -m(m-2) \om_m \delta_x$$ 
where $D = \{(x,x) \mid x \in M\}$ is the diagonal and $\om_m$ is the volume of the unit $m$-ball. By fixing a pole $x\in M$, we may understand $G$ as a smooth one-variable function on $M\backslash\{x\}$ given by $G(y) = G(x,y)$. In particular,  $G = r^{2-m}$ when $M = \C^n$ with the flat metric. Although the positive Green function is nonunique, one can always choose the minimal one. Given the presence of a gradient estimate, it is natural to ask whether one can extend it to a full Hessian estimate on the Green function, which would be an elliptic analogue of Theorem \ref{hamilton_thm}. Indeed, a matrix Harnack inequality holds on manifolds with nonnegative sectional curvature and parallel Ricci curvature \cite{P}. In fact the nonnegative sectional curvature condition can be slightly weakened in this case; for the precise statement of the assumptions, we refer the reader to \cite{P}.

The aim of this paper is to prove the following complex geometry analogue of the matrix inequality for the Green function. 

\begin{thm}\label{mainthm}
Let $M$ be a nonparabolic complete K\"ahler manifold of complex dimension $ n\geq$ 2, and $p \in M$. Suppose that $M$ has nonnegative holomorphic bisectional curvature. Let $G$ be the minimal positive Green function with pole $x \in M$. Suppose that for some $C \in \R$,
\begin{align} \label{hess_at_pole}
	G_{\af \overline{\bt}} +\frac{2n}{2-2n} \cdot \frac{G_\af G_{\overline{\bt}}}{G}  \geq C G^{\frac{2n}{2n-2}} g_{\af \overline{\bt}} 
\end{align}
on a neighborhood of $x$, and that for some $D \in \R$,
\begin{align} \label{inf1}
	G_{\af \overline{\bt}} +\frac{2n}{2-2n} \cdot \frac{G_\af G_{\overline{\bt}}}{G}  \geq D G^{\frac{2n}{2n-2}} g_{\af \overline{\bt}} 
\end{align}
on $M \backslash \{x\}$. Then the Hessian estimate (\ref{hess_at_pole}) holds on $M \backslash \{x\}$, no matter how small $D$ is. In other words, the constant $C$ given near the pole $x$ propagates to all of $M$.
\end{thm}

The theorem can be stated with error terms for holomorphic bisectional curvature merely bounded below, instead of nonnegative. The notable advantage of the complex case over the Riemannian case is that we only need to assume nonnegativity of the holomorphic bisectional curvature, which is a much weaker condition than nonnegative sectional curvature and parallel Ricci curvature. 

The coefficient $\displaystyle{\frac{2n}{2-2n}}$ might seem arbitrary at first sight. In fact it is the most natural choice, since the tensor on the left hand side is the canonical one to consider, in the sense that if $M = \C^n$ then equality holds in (\ref{hess_at_pole}) with $C = (2n-2)$,
\begin{align} \label{hess_Cn}
G_{\af \overline{\bt}} +\frac{2n}{2-2n} \cdot \frac{G_\af G_{\overline{\bt}}}{G} = (2n-2) G^{\frac{2n}{2n-2}} g_{\af \overline{\bt}}.	
\end{align}

However it is also possible to formulate Theorem \ref{mainthm} for not only $\displaystyle{\frac{2n}{2n-2}}$, but any other negative coefficient, as stated in the following theorem.
 
\begin{thm}\label{mainthm2}
Let $C \leq -1$. Let $M$ be a nonparabolic complete K\"ahler manifold of complex dimension $ n\geq$ 2, and $p \in M$. Suppose that $M$ has nonnegative holomorphic bisectional curvature. Let $G$ be the minimal positive Green function with pole $x \in M$. Suppose that for some $D \in \R$,
\begin{align} \label{hess_at_pole}
	G_{\af \overline{\bt}} + C \cdot \frac{G_\af G_{\overline{\bt}}}{G}  \geq D G^{\frac{2n}{2n-2}} g_{\af \overline{\bt}} 
\end{align}
on a neighborhood of $x$, and that for some $E, F \in \R$,
\begin{align} \label{inf2}
	G_{\af \overline{\bt}} + E \cdot \frac{G_\af G_{\overline{\bt}}}{G}  \geq F G^{\frac{2n}{2n-2}} g_{\af \overline{\bt}} 
\end{align}
on $M \backslash \{x\}$. Then the Hessian estimate (\ref{hess_at_pole}) holds on $M \backslash \{x\}$, no matter how small $F$ is. In other words, the constant $D$ given near the pole $x$ propagates to all of $M$.
\end{thm}

Theorems \ref{mainthm} and \ref{mainthm2} can be understood as an elliptic version of \cite{CN}, where Cao and Ni proved a complex analogue of Hamilton's matrix inequality for the heat equation on Riemannian manifolds. There are Matrix Li-Yau-Hamilton estimates for other parabolic equations as well, such as for the Ricci flow \cite{H1}, K{\"a}hler-Ricci flow \cite{N1}, and mean curvature flow \cite{H3}.

\vspace{0.2cm}

It is also worthwhile to view these theorems as a Hessian comparison theorem for the distance-like function $b$, defined to be $b = G^{\frac{1}{2-2n}}$. Then (\ref{hess_Cn}) is equivalent to $\Hess_{b^2} \leq 2g$, which is an equality precisely on $\C^n$. For the distance function itself, there is an abundance of Hessian comparison theorems both on Riemannian manifolds on K\"ahler manifolds. To name just a few among the K\"ahler cases, we note that in \cite{GW} and \cite{CN}, Hessian comparison theorems were proved for nonnegative holomorphic bisectional curvature, and in \cite{LW} for nonnegative holomorphic sectional curvature. Among many applications of the Hessian comparison theorems, we note that Liu \cite{L} applied \cite{LW} to prove that a complete K\"ahler manifold satisfies the three circle theorem if and only if it has nonnegative holomorphic sectional curvature. We expect that Theorems \ref{mainthm} and \ref{mainthm2} would also be useful in the analysis on manifolds with holomorphic bisectional curvature bounded below, considering the advantage of $b$ over the distance function that it satisfies an elliptic equation and is smooth. Indeed, $b$ is used as an almost splitting function for manifolds of nonnegative Ricci curvature in Cheeger-Colding theory \cite{CC}, \cite{CC1}. $\Hess_{b^2}$ is a term that comes up naturally in the careful analysis of weighted integrals measuring the Gromov-Hausdorff distance between such manifolds and their tangent cones \cite{CM}. 

\vspace{0.2cm}

In the rest of this paper we will prove Theorems \ref{mainthm} and \ref{mainthm2}. Before embarking on the proof, in Section 2 we will recall some facts about the positive Green function, then provide the proof in Section 3.

\newpage

\section{Nonparabolicty and the gradient estimate for Green function}

 The results in this section hold for Riemannian manifolds, not necessarily K\"ahler manifolds. 

Let $M$ be a complete Riemannian manifold of dimension $m$, and $x \in M$. If $M$ has nonnegative Ricci curvature, Varopoulos \cite{V} showed that $M$ is non-parabolic if and only if $$\int_s^\infty \frac{t}{\mathrm{Vol}\left(B_x(t)\right)} dt<\infty$$ for some $s>0$. If $M$ is nonparabolic, then by the work of Li and Tam \cite{LT} and Gilbarg and Serrin \cite{GS}, there exists a unique minimal positive symmetric Green function $G$ with pole at $x$ such that $G=O(r^{2-m})$, where $r$ is the distance from $x$. As in the introduction, define the function $b$ by
\begin{equation*}
	b=G^\frac{1}{2-m}.
\end{equation*} 

To prove  Theorem \ref{mainthm2}, we will need the following gradient estimate for $b$.

\begin{thm}[\cite{C}, Theorem 3.1] 
Let $M^m$ be a nonparabolic Riemannian manifold with nonnegative Ricci curvature of dimension $ m \geq 3$. Then $$|\nab b| \leq 1,$$
Or equivalently,
\begin{equation}\label{gradest}
|\nab G| \leq (m-2) G^{\frac{m-1}{m-2}}.
\end{equation}
\end{thm}

\section{Proof of Theorems \ref{mainthm} and \ref{mainthm2}}

As in \cite{CN} and \cite{H1}, the main ingredient in the proof is the matrix maximum principle of Hamilton. The difference with the parabolic case is that we work on the spacial domain $\{0<r<R\}$ then take $R \to \infty$. In some sense, $r$ plays the role of $t$ in our proof.

Let $M$ be a K\"ahler manifold of dimension $n$. Let us work in holomorphic normal coordinates, so that the Riemannian Laplacian is given by $$\Dt = \frac{1}{2} \left(\nab_\gm \nab_{\overline{\gm}} + \nab_{\overline{\gm}}\nab_\gm \right).$$

Let $s \in \R$. By a direct calculation using that $\Dt G = 0$, we obtain the following expression of $\Dt G^{s}$.

\begin{equation*} 
\Dt G^{s} = s(s-1)G^{s-2} G_{\gm} G_{\overline{\gm}}.
\end{equation*}

In particular, taking $\displaystyle s = \frac{2n}{2n-2}$,

\begin{equation*}
\Dt G^{\frac{2n}{2n-2}} = \frac{4n}{(2n-2)^2} G^{-\frac{2n-4}{2n-2}}G_\gm G_{\overline{\gm}}.
\end{equation*}

It is convenient to define the tensor $B$ by
\begin{equation*}
	B_{\af \overline{\bt}} = \frac{G_{\af} G_{\overline{\bt}}}{G}.
\end{equation*}

Then $B$ is nonnegative definite with eigenvalues $\displaystyle \frac{G_{\gm} G_{\overline{\gm}}}{G}$ and 0.

Fix $\af, \bt \in \{1, \cdots, n\}$. To calculate $\Dt G_{\af \overline{\bt}}$ and $\Dt B_{\af \overline{\bt}}$, we take note of the following standard identities on K\"ahler manifolds, again in holomorphic normal coordinates. Here, the subscripts simply indicate taking the directional derivatives repeatedly.

\begin{align*}
G_{\af \gm \overline{\gm}}& = G_{\gm \overline{\gm} \af} + R_{\af \overline{\dt}} G_{\dt},\\
G_{\af \overline{\gm} \gm} &= G_{\overline{\gm} \gm \af}, \\
G_{\overline{\bt} \gm \overline{\gm}}& = G_{\gm \overline{\gm} \overline{\bt}}, \\
G_{\overline{\bt} \overline{\gm} \gm} &= G_{\overline{\gm} \gm \overline{\bt}} + R_{\dt \overline{\bt}} G_{\overline{\dt}}.
\end{align*}

Using that $\Dt G = 0$, we obtain that

\begin{align} \label{Dt_G}
\Dt (G_{\af }) &= \frac{1}{2}R_{\af \overline{\dt}} G_{\dt}, \nonumber \\
\Dt (G_{\overline{\bt}}) &= \frac{1}{2}R_{\dt \overline{\bt}} G_{\overline{\dt}}. 
\end{align}

One cal also calculate $\Dt  (G_{\af \overline{\bt}})$ in a straightforward manner,

\begin{equation} \label{Dt_G2}
\Dt (G_{\af \overline{\bt}}) = -R_{\af \overline{\bt} \dt \overline{\gm}} G_{\gm \overline{\dt}} + \frac{1}{2}R_{\gm \overline{\bt}} G_{\af \overline{\gm}} + \frac{1}{2}R_{\af \overline{\gm}} G_{\gm \overline{\bt} }.	
\end{equation}

(\ref{Dt_G2}) can be obtained by a standard calculation similar to the appendix \cite{P}. One can also calculate $\Dt u_{\af \overline{\bt}}$ for $u$ a solution to the heat equation as in \cite{NT}, Lemma 2.1, then observe that $G$ is a stationary solution.

We also calculate $\Dt B_{\af \overline{\bt}}$.

\begin{align} \label{Dt_B}
\Dt B_{\af \overline{\bt}} = &\frac{\Dt(G_\af) G_{\overline{\bt}}+\Dt(G_{\overline{\bt}}) G_\af}{G} + \frac{G_{\af \gm} G_{\overline{\bt} \overline{\gm}} + G_{\af \overline{\gm}} G_{\overline{\bt} \gm}}{G} \nonumber  \\ 
	&- \frac{G_{\af \overline{\gm}}G_{\overline{\bt}}G_{\gm}+G_{\af \gm}G_{\overline{\bt}}G_{\overline{\gm}}+G_{\overline{\bt} \gm}G_{\af}G_{\overline{\gm}}+G_{\overline{\bt} \overline{\gm}}G_{\af}G_{\gm}}{G^2} +2 \frac{G_\af G_{\overline{\bt}} G_\gm G_{\overline{\gm}}}{G^3} \nonumber  \\
	=&\frac{R_{\af \overline{\dt}} G_{\dt} G_{\overline{\bt}}+R_{\dt \overline{\bt}} G_{\overline{\dt}} G_\af}{2G} + \frac{G_{\af \gm} G_{\overline{\bt} \overline{\gm}} + G_{\af \overline{\gm}} G_{\overline{\bt} \gm}}{G}  \nonumber \\
	&- \frac{G_{\af \overline{\gm}}G_{\overline{\bt}}G_{\gm}+G_{\af \gm}G_{\overline{\bt}}G_{\overline{\gm}}+G_{\overline{\bt} \gm}G_{\af}G_{\overline{\gm}}+G_{\overline{\bt} \overline{\gm}}G_{\af}G_{\gm}}{G^2} +2 \frac{G_\af G_{\overline{\bt}} G_\gm G_{\overline{\gm}}}{G^3} \nonumber \\
	=&\frac{R_{\af \overline{\dt}} G_{\dt} G_{\overline{\bt}}+R_{\dt \overline{\bt}} G_{\overline{\dt}} G_\af}{2G} + \frac{ G_{\af \overline{\gm}} G_{\gm \overline{\bt} }}{G} + \frac{1}{G}\left(G_{\af \gm} -\frac{G_{\af} G_{\gm}}{G}\right)\left(G_{\overline{\bt} \overline{\gm}} - \frac{G_{\overline{\bt}}G_{\overline{\gm}}}{G} \right) \nonumber \\
	&- \frac{G_{\af \overline{\gm}}G_{\overline{\bt}}G_{\gm}+G_{\gm \overline{\bt} }G_{\af}G_{\overline{\gm}}}{G^2} +  \frac{G_\af G_{\overline{\bt}} G_\gm G_{\overline{\gm}}}{G^3}.
\end{align}

\begin{proof}[Proof of Theorem \ref{mainthm2}]
For $C < -1$ and $D \in \R$, define the tensor $T$ by 

\begin{equation}
T_{\af \overline{\bt}} = G_{\af \overline{\bt}} + C \cdot \frac{G_\af G_{\overline{\bt}}}{G} - D \cdot  G^{\frac{2n}{2n-2}} g_{\af \overline{\bt}}.
\end{equation}

Our goal is to show that under the assumptions of Theorem \ref{mainthm2}, $T_{\af \overline{\bt}} \geq 0$. Let us consider the domain $\{\ep < r< R\}$, where we will take $\ep \to 0$ and $R \to \infty$. To show that $T_{\af \overline{\bt}} \geq 0$, by the maximum principle it is enough to check that $$\Dt T_{\af \overline{\bt}}  \leq 0 \text{ on } \{\ep < r< R\},$$ and that $$T_{\af \overline{\bt}} \geq -f(r)$$ for some continuous $f: (0,\infty) \to (0, \infty)$ with $\displaystyle \lim_{r \to 0} f(r) = 0$ and $\displaystyle \lim_{r \to \infty} f(r) = 0$. In fact, for the maximum principle argument (see for instance \cite{H2}, \cite{CN}, \cite{P}) it suffices to check that $\Dt (T_{\af \overline{\bt}}) V_\af V_{\overline{\bt}} \leq 0$ when $T$ is nonnegative definite and $V$ is a null vector of $T$ so that $\displaystyle T_{\gm \overline{\dt}} V_{\gm} =T_{\gm \overline{\dt}} V_{\overline{\dt}} = 0$, instead of all vectors.

Using equations (\ref{Dt_G}), (\ref{Dt_G2}) and (\ref{Dt_B}), and simplifying the calculation noting that by the definition of Ricci curvature,

$$-R_{\af \overline{\bt} \dt \overline{\gm}} g_{\gm \overline{\dt}}+ \frac{1}{2}R_{\gm \overline{\bt}}g_{\af \overline{\gm}} + \frac{1}{2}R_{\af \overline{\gm}} g_{\gm \overline{\bt}} = 0,$$
it is now straightforward to calculate $\Dt T_{\af \overline{\bt}}$,

\begin{align} \label{Dt_T}
\Dt (T_{\af \overline{\bt}})
&= - R_{\af \overline{\bt} \dt \overline{\gm}} T_{\gm \overline{\dt}} + \frac{1}{2}R_{\gm \overline{\bt}}T_{\af \overline{\gm}}+ \frac{1}{2}R_{\af \overline{\gm}}T_{\gm \overline{\bt}} +\frac{C}{G}\left(G_{\af \gm} -\frac{G_{\af} G_{\gm}}{G}\right)\left(G_{\overline{\bt} \overline{\gm}} - \frac{G_{\overline{\bt}}G_{\overline{\gm}}}{G} \right)  \nonumber \\
& + \frac{C}{G} T_{\af \overline{\gm}} T_{\gm \overline{\bt}}+ CR_{\af \overline{\bt} \dt \overline{\gm}} B_{\gm \overline{\dt}}  + CD G^{\frac{2}{2n-2}} T_{\af \overline{\bt}} + CD^{2} G^{\frac{2n+2}{2n-2}} g_{\af \overline{\bt}} \nonumber \\
&-\frac{C^2+C}{G}\left(T_{\af \overline{\gm}} B_{\gm \overline{\bt}} + T_{\gm \overline{\bt}} B_{\af \overline{\gm}}\right) -\frac{(C^2+C)D}{G}G^{\frac{2n}{2n-2}} \left(g_{\af \overline{\gm}} B_{\gm \overline{\bt}} + g_{\gm \overline{\bt}} B_{\af \overline{\gm}}  \right) \nonumber \\
& + \frac{C^3 + 2C^2 + C}{G}B_{\af \overline{\gm}} B_{\gm \overline{\bt}} - D \cdot \frac{4n}{(2n-2)^2} G^{-\frac{4n-6}{2n-2}}B_{\gm \overline{\gm}} g_{\af \overline{\bt}}.	
\end{align}

Now we make some observations in order to check that the terms on the left hand side of (\ref{Dt_T}) are nonpositive on a null vector $V$ of $T$. For the first three terms, we observe that

 \begin{align*}
\left(- R_{\af \overline{\bt} \dt \overline{\gm}} T_{\gm \overline{\dt}} + \frac{1}{2}R_{\gm \overline{\bt}}T_{\af \overline{\gm}}+ \frac{1}{2}R_{\af \overline{\gm}}T_{\gm \overline{\bt}} 	\right) V_{\af} V_{\overline{\bt}} = -R_{\af \overline{\bt} \dt \overline{\gm}} T_{\gm \overline{\dt}} V_{\af} V_{\overline{\bt}} \leq 0,
 \end{align*}
since $T$ is nonnegative definite, and $M$ has nonnegative holomorphic bisectional curvature, hence nonnegative Ricci curvature.

For the remainder of the first row, since $C<0$, we obtain that $$\frac{C}{G}\left(G_{\af \gm} -\frac{G_{\af} G_{\gm}}{G}\right)\left(G_{\overline{\bt} \overline{\gm}} - \frac{G_{\overline{\bt}}G_{\overline{\gm}}}{G} \right) \leq 0,$$ and $$\displaystyle \frac{C}{G} T_{\af \overline{\gm}} T_{\gm \overline{\bt}} \leq 0.$$

By the assumption that $M$ has nonnegative holomorphic bisectional curvature, $$CR_{\af \overline{\bt} \dt \overline{\gm}} B_{\gm \overline{\dt}} V_\af V_{\overline{\bt}} = \frac{C}{G}R_{\af \overline{\bt} \dt \overline{\gm}}V_\af V_{\overline{\bt}} G_{\gm} G_{\overline{\dt}}\leq 0.$$

We have that $$\displaystyle CD G^{\frac{2}{2n-2}} T_{\af \overline{\bt}}V_{\af} V_{\overline{\bt}} = 0$$ 
since $V$ is a null vector of $T$. Since $CD \leq 0$ and $T $ is nonnegative definite, we obtain that  It is also obvious that $\displaystyle CD^2 G^{\frac{2n+2}{2n-2}} g_{\af \overline{\bt}}V_{\af} V_{\overline{\bt}} \leq 0$.

Next, $C \leq -1$ implies $C^2 + C \geq 0$. Using that $B$ is nonnegative definite and that $V$ is a null vector of $T$, the third row in (\ref{Dt_T}) also becomes nonpositive on $(V^\af)$,

\begin{equation*}
\left(-\frac{C^2+C}{G}\left(T_{\af \overline{\gm}} B_{\gm \overline{\bt}} + T_{\gm \overline{\bt}} B_{\af \overline{\gm}}\right) -\frac{(C^2+C)D}{G}G^{\frac{2n}{2n-2}} \left(g_{\af \overline{\gm}} B_{\gm \overline{\bt}} + g_{\gm \overline{\bt}} B_{\af \overline{\gm}}  \right) \right)	 V_{\af} V_{\overline{\bt}} \leq 0,
\end{equation*}

The last row of (\ref{Dt_T}) is obviously nonpositive as well, since $C^3 + 2C^2 + C \leq 0$ and $D \geq 0$. Combining these observations, we conclude that $\displaystyle \Dt T_{\af \overline{\bt} } V_{\af} V_{\overline{\bt}} \leq 0$.

It only remains to check the boundary. By the hypothesis of the theorem, we already have that $T \geq 0$ as $r \to 0$, and that 
\begin{equation} \label{CEFD}
T \geq (C-E)\frac{G_\af G_{\overline{\bt}}}{G} + (F - D)G^{\frac{2n}{2n-2}}g_{\af\overline{\bt}}\end{equation}
 on $M$. Since $G = O(r^{2-2n})$,	
 we have that $G^{\frac{2n}{2n-2}} = O(r^{-2n})$, so $G^{\frac{2n}{2n-2}} \to 0$ as $r \to \infty$. Also by the gradient estimate (\ref{gradest}), $$\left| \frac{G_\af G_{\overline{\bt}}}{G} \right| \leq C(n) G^{\frac{2n}{2n-2}} = O(r^{-2n}),$$  Where $C(n)$ is a dimensional constant. Thus we conclude that the right hand side of (\ref{CEFD}) tends to 0 as $r \to \infty$. Theorem \ref{mainthm2} now follows by the maximum principle.
\end{proof}

\vspace{0.2cm}

\begin{proof}[Proof of Theorem \ref{mainthm}]
Theorem \ref{mainthm} follows immediately from Theorem \ref{mainthm2} by replacing $C, D, E$, and $F$ with $\displaystyle \frac{2n}{2-2n}, C, \frac{2n}{2-2n}$, and $D$, resplectively.
\end{proof}

\vspace{0.2cm}
\begin{remark}
It would be desirable to drop the boundary conditions in both theorems, especially the assumptions at infinity (\ref{inf1}), (\ref{inf2}). In the parabolic case \cite{CN} the noncompact case can be dealt with using Li-Yau gradient estimate and the Hopf maximum principle. However, in our problem a naive application of this strategy leads to undesirable signs in the calculation, hence does not work. It is an interesting problem whether these assumptions can be removed by some other way.
\end{remark}

\end{document}